\documentclass[12pts]{article}
\usepackage{amsmath,amsfonts,amssymb,amsthm,mathtools}
\usepackage{microtype}
\usepackage{lipsum}
\usepackage{hyperref}
\usepackage{indentfirst}
\hypersetup{colorlinks, linkcolor={red}, citecolor={blue}, urlcolor={black}}
\usepackage{comment}
\usepackage[margin=1.25in]{geometry}
\usepackage{enumerate}

\allowdisplaybreaks

\newcommand\ceil[1]{\left\lceil#1\right\rceil}
\newcommand\floor[1]{\left\lfloor#1\right\rfloor}

\newtheorem{theorem}{Theorem}[section]
\newtheorem{lemma}[theorem]{Lemma}

\newtheorem{corollary}[theorem]{Corollary}
\newtheorem{definition}[theorem]{Definition}

\newtheorem{question}[theorem]{Question}

\title{Distribution of colours in rainbow $H$-free colourings}
\author{Zhuo Wu\thanks{Mathematics Institute and DIMAP, University of Warwick, UK. Email: {\tt Zhuo.Wu@warwick.ac.uk}. Supported by the Warwick Mathematics Institute Centre for Doctoral Training and funding from University of Warwick.}
\and Jun Yan\thanks{Mathematics Institute, University of Warwick, UK. Email: {\tt jun.yan@warwick.ac.uk}. Supported by the Warwick Mathematics Institute Centre for Doctoral Training and funding from the UK EPSRC (Grant number: EP/W523793/1).}}
\date{}

\begin{document}

\maketitle

\begin{abstract}
An edge colouring of $K_n$ with $k$ colours is a \emph{Gallai $k$-colouring} if it does not contain any rainbow triangle. Gy\'arf\'as, P\'alv\"olgyi, Patk\'os and Wales proved that there exists a number $g(k)$ such that $n\geq g(k)$ if and only if for any colour distribution sequence $(e_1,\cdots,e_k)$ with $\sum_{i=1}^ke_i=\binom{n}{2}$, there exist a Gallai $k$-colouring of $K_n$ with $e_i$ edges having colour $i$. They also showed that $\Omega(k)=g(k)=O(k^2)$ and posed the problem of determining the exact order of magnitude of $g(k)$. Feffer, Fu and Yan improved both bounds significantly by proving $\Omega(k^{1.5}/\log k)=g(k)=O(k^{1.5})$. We resolve this problem by showing $g(k)=\Theta(k^{1.5}/(\log k)^{0.5})$.

Moreover, we generalise these definitions by considering rainbow $H$-free colourings of $K_n$ for any general graph $H$, and the natural corresponding quantity $g(H,k)$. We prove that $g(H,k)$ is finite for every $k$ if and only if $H$ is not a forest, and determine the order of $g(H,k)$ when $H$ contains a subgraph with minimum degree at least 3.
\end{abstract}

\section{Introduction}
An edge colouring of $K_n$ using $k$ colours is a map $c: E(K_n)\to [k]$, where the colour of an edge $uv$ is $c(uv)$. Given a simple graph $H$ and an edge colouring of $K_n$, we say that this colouring contains a \emph{rainbow copy} of $H$ if there exists a subgraph of $K_n$ isomorphic to $H$, whose edges all have different colours. Otherwise, we say this colouring is \emph{rainbow $H$-free}. 

For non-negative integers $e_1,\cdots,e_k$, we say the sequence $(e_1,\cdots,e_k)$ is \emph{$n$-good} iff $\sum_{i=1}^ke_i=\binom{n}{2}$. An edge colouring of the complete graph $K_n$ using $k$ colours is said to have \emph{colour distribution sequence} $(e_1,\cdots,e_k)$ if exactly $e_i$ edges have colour $i$ for every $i\in[k]$. In this case, we also say that this is an \emph{$(e_1,\cdots,e_k)$-colouring} of $K_n$.

In \cite{GS}, Gy\'arf\'as and Simonyi initiated the study of rainbow $K_3$-free colourings of $K_n$. They named these colourings Gallai colourings, referring to Gallai's related work \cite{G} on comparability graphs. Since then, many Ramsey-type results \cite{GSSS,GS,MS} and enumeration results \cite{BBH, BBMS, BL} for Gallai colourings have been obtained. 

A new avenue of research on the distribution of colours in Gallai colourings was proposed by Gy\'arf\'as, P\'alv\"olgyi, Patk\'os and Wales in \cite{GPPW}. They proved the following theorem, which says that as long as we have sufficiently many vertices, any colour distribution sequence can be realised as a Gallai colouring.

\begin{theorem}[\cite{GPPW}]
For every integer $k\geq 2$, there exists an integer $N$ such that for all $n\geq N$ and all non-negative integers $e_1,\cdots,e_k$ satisfying $\sum_{i=1}^ke_i=\binom{n}{2}$, there exists a Gallai colouring of $K_n$ with colour distribution sequence $(e_1,\cdots,e_k)$.
\end{theorem}

In \cite{GPPW}, Gy\'arf\'as, P\'alv\"olgyi, Patk\'os and Wales denoted the smallest such integer $N$ by $g(k)$. They also gave the estimate $2k-2\leq g(k)\leq 8k^2+1$, and posed the problem of determining the exact order of magnitude of $g(k)$. The gap was closed considerably when Feffer, Fu and Yan \cite{FFY} proved that $\Omega(k^{1.5}/\log k)\leq g(k)\leq O(k^{1.5})$. The first main result of this paper resolves this problem. 

\begin{theorem}\label{triangle}
$g(k)=\Theta(k^{1.5}/(\log k)^{0.5})$.
\end{theorem}

Motivated by this result, we study a generalised version of this problem. We first generalise the quantity $g(k)$ to $g(H,k)$ as follows. 

\begin{definition}
Let $H$ be a graph and let $k$ be a positive integer. Define $g(H,k)$ to be the smallest integer $N$, such that for all $n\geq N$ and all non-negative integers $e_1,\cdots,e_k$ satisfying $\sum_{i=1}^ke_i=\binom{n}{2}$, there exists a rainbow $H$-free colouring of $K_n$ with colour distribution sequence $(e_1,\cdots,e_k)$. If such $N$ does not exist, define $g(H,k)=\infty$.
\end{definition}

Note that setting $H=K_3$ recovers the definition of $g(k)$. Moreover, we have the following simple property of $g(H,k)$.

\begin{lemma}\label{subgraph}
If $H_1$ is a subgraph of $H_2$, then $g(H_1,k)\geq g(H_2,k)$.
\end{lemma}
\begin{proof}
This follows from the observation that any rainbow $H_1$-free graph is also rainbow $H_2$-free.
\end{proof}

The natural questions of interest regarding $g(H,k)$ are the following.

\begin{question}
For what graphs $H$ and integers $k$ is $g(H,k)$ finite? If $g(H,k)$ is finite, What is the order of magnitude for $g(H,k)$ in terms of $k$?
\end{question}

It turns out that the behaviour of $g(H,k)$ depends a graph parameter called degeneracy. We recall this definition below. 

\begin{definition}
Let $H$ be a graph and let $k$ be a positive integer.
\begin{itemize}
    \item $H$ is $k$-degenerate if every subgraph $H'$ of $H$ has a vertex of degree at most $k$.
    \item The degeneracy of $H$ is the smallest positive integer $k$ for which $H$ is $k$-degenerate. 
\end{itemize}
\end{definition}

As two examples, 1-degenerate graphs are exactly forests, and every graph with degeneracy at least $k$ contains a subgraph of minimum degree at least $k$. With this definition in mind, the results we will prove in this paper are the following.

\begin{theorem}\label{tree1}
If $H$ is a forest, or equivalently have degeneracy 1, then $g(H,k)=\infty$ for all sufficiently large $k$.
\end{theorem}

\begin{theorem}\label{degen3}
If $H$ is a graph on $m$ vertices with degeneracy at least $3$, then $g(H,k)=\Theta_m(k)$.
\end{theorem}

\begin{theorem}\label{degen2}
If $H$ is a graph on $m$ vertices with degeneracy $2$, then \[\Omega_m(k)\leq g(H,k)\leq O(k^{1.5}/(\log k)^{0.5}).\]
\end{theorem}

The paper is organised as follows.
We prove Theorem~\ref{tree1}, Theorem~\ref{degen3} and Theorem~\ref{degen2} in Section~\ref{tree}, Section~\ref{d3} and Section~\ref{upper}, respectively. Note that the upper bound in Theorem~\ref{degen2} also gives the upper bound in Theorem~\ref{triangle}. Finally, we prove the lower bound of Theorem~\ref{triangle} in Section~\ref{t}.

\section{Trees}\label{tree}
In this section, we prove Theorem \ref{tree1}. Note that any forest is the subgraph of a tree, therefore, by Lemma \ref{subgraph}, it suffices to prove Theorem \ref{tree1} in the case when $H$ is a tree. In fact, we will prove the following stronger result.

\begin{theorem}\label{tree2}
For every integer $m\geq 2$, there exists a constant $D=D(m)$ such that the following holds: If $H$ be any tree on $m$ vertices, then any edge colouring of $K_n$ such that every colour is used to colour at most $\frac1D\binom{n}{2}$ edges contains a rainbow copy of $H$.
\end{theorem}

Note that the condition in Theorem \ref{tree2} implies implicitly that $\frac1D\binom{n}{2}\geq 1$, or equivalently $n\geq\Omega(\sqrt D)$.

We first show that the Theorem~\ref{tree2} implies Theorem~\ref{tree1}.

\begin{proof}[Proof of Theorem \ref{tree1}]
Let $D=D(m)$ be the constant in Theorem~\ref{tree2}. Let $k\geq\lfloor 2D \rfloor$ be an integer. For any sufficiently large integer $n$, assume that $\binom{n}{2}=qk+r$, where $0\le r\le k-1$. Set $e_1=e_2=\cdots=e_{k-r}=q$ and $e_{k-r+1}=\cdots=e_k=q+1$. Then $(e_1,\cdots,e_k)$ is $n$-good and $$e_i\le 1+\frac1k\binom{n}{2}\le \frac1D\binom{n}{2}$$
for all $i\in[k]$. Hence, by Theorem~\ref{tree2}, any colouring of $K_n$ with colour distribution sequence $(e_1,\cdots,e_k)$ contains a rainbow copy of $H$. From definition, this proves $g(H,k)=\infty$.
\end{proof}

We will prove Theorem \ref{tree2} with $D(m)=(6m)^{6m}$. Given an edge colouring $c$ of $K_n$ and any vertex $v\in V$, let
\[C(v)=\bigcup_{v'\in V\setminus\{v\}}\{c(vv')\}\]
be the set of colours assigned to edges adjacent to $v$. The following technical lemma shows that at least half of the vertices satisfy $|C(v)|\geq2m+1$.
\begin{lemma}\label{smallcolour}
Let $m\geq 2$ be an integer and let $D=(6m)^{6m}$. Let $c$ be an edge colouring of $K_n$ such that every colour is used to colour at most $\frac1D\binom{n}{2}$ edges. Then there are at most $\frac12n$ vertices $v$ such that $|C(v)|\le 2m$.
\end{lemma}

\begin{proof}
We prove this lemma by contradiction. Assume that there is a set $V_0$ of at least $\frac12n$ vertices $v\in V(K_n)$ such that $|C(v)|\le 2m$. We show by induction that for every integer $0\leq j\leq 2m$, we can find a subset $V_j$ of $V_0$ such that

\[ \left|\bigcap_{v\in V_j}C(v)\right|\ge j\text{\quad and\quad}|V_j|\ge \frac{n}{2(6m)^j}.\]

The base case follows from the definition of $V_0$. Suppose now we have picked $V_j$ as required for some $0\leq j\leq 2m-1$. Without loss of generality, assume that
\[\{1,2,\dots,j\}\subseteq \bigcap_{v\in V_j}C(v).\]
From the assumption, there are at most $\frac{j{\binom n2}}{D}$ edges coloured with one of the colours $1,2,\cdots,j$. Hence, by pigeonhole, there exists a vertex $u\in V_j$ such that there are at most $\frac{2j\binom{n}{2}}{D|V_j|}$ edges adjacent to $u$ with colour $1,2,\cdots,j$. Note that for any $u'\in V_j\setminus\{u\}$, the colour $c(uu')$ of the edge $uu'$ must belong to the set $C(u)$, which has size at most $2m$. By pigeonhole again, there exists some colour $q\in C(u)\setminus[j]$ such that at least
\[\frac{|V_j|-1-\frac{2j\binom{n}{2}}{D|V_j|}}{2m-j}\]
vertices $u'\in V_j\setminus\{u\}$ satisfy $c(uu')=q$. Set
\[V_{j+1}=\{u'\mid u'\in V_j\setminus\{u\}\text{ and } c(uu')=q.\}\]
We verify that $V_{j+1}$ works. 

Indeed, the first condition holds as
\[\{1,2,\cdots,j,q\}\subseteq \bigcap_{v\in V_{j+1}}C(v).\]
For the second condition, we have
\begin{align*}
|V_{j+1}|&\ge \frac{|V_j|-1-\frac{2j\binom{n}{2}}{D|V_j|}}{2m-j}\\
&\ge \frac{|V_j|-1-\frac{2j(n-1)(6m)^j}{D}}{2m} && \text{by }|V_j|\ge \frac{n}{2(6m)^j} \\ 
&\ge \frac{|V_j|-1-\frac n{(6m)^{2m}}}{2m} && \text{by }D=D(m)=(6m)^{6m}\\
&\ge \frac{|V_j|}{6m}\ge \frac{n}{2(6m)^{j+1}}, && \text{by }|V_j|\ge \frac{n}{2(6m)^j} 
\end{align*}
as required. Hence, we can pick the sets $V_0,V_1,\cdots,V_{2m}$ by induction. 

Note that for any $u\in V_{2m}$, we have
$$2m\geq|C(u)|\geq\left|\bigcap_{v\in V_{2m}}C(v)\right|\geq 2m.$$
Hence, there exists a set $C$ of size $2m$ such that $C(u)=C$ for all $u\in V_{2m}$. In particular, for any two vertices $u,u'\in V_{2m}$, the colour of the edge $uu'$ must belong to this set $C$. Therefore, we have
\[\binom{|V_{2m}|}{2}\le 2m\frac{\binom{n}{2}}{D}.\]
However, 
\[\binom{|V_{2m}|}{2}\ge \frac{1}{4}|V_{2m}|^2\ge \frac{n^2}{16(6m)^{4m}}\ge \frac{n^2}{(6m)^{5m}}> 2m\frac{\binom{n}{2}}{D},\]
a contradiction. This completes the proof of Lemma \ref{smallcolour}.
\end{proof}

Now we can prove Theorem \ref{tree2}.

\begin{proof}[Proof of Theorem \ref{tree2}]
We will show that $D(m)=(6m)^{6m}$ works.
Proceed by induction on $m$. The case $m=2$ is trivial. Suppose we have already established the theorem for all trees with fewer than $m$ vertices. Since $H$ is a tree, it has a leaf $v_0$. Let the vertices in $H$ be $v_0,v_1,\cdots, v_{\ell-1}$, and suppose the unique edge adjacent to $v_0$ is $v_0v_1$. Let $H'=H-v_0$. 

Let $G'$ be the subgraph of $K_n$ induced by the set of vertices $v$ satisfying $|C(v)|\ge 2m+1$, then $G'$ contains at least $\frac12n$ vertices by Lemma \ref{smallcolour}. From assumption, each colour is used at most
\[\frac{\binom{n}{2}}{D(m)}\le \frac{\binom{|v(G')|}{2}}{D(m-1)}\]
times in $G'$. Therefore, we may apply induction hypothesis to find a rainbow copy of $H'$ in $G'$. Suppose vertex $v_i$ in $H'$ is embedded to vertex $u_i$ in $G'$ for all $1\le i\le m-1$. Note that

\[\left|\bigcup_{v\in V\setminus\{u_2,\cdots,u_{m-1}\}}\{c(u_1v)\}\right|\ge \big|C(u_1)\big|-(m-2)\geq m-1.\]
Hence, we can find a vertex $u_0\in V\setminus\{u_1,u_2,\cdots,u_{k-1}\}$ such that the edge $u_0u_1$ has a different colour to all $m-2$ edges in this rainbow copy of $H'$. It follows that there exists a rainbow copy of $H$ in $K_n$ with vertices $u_0,u_1,\cdots,u_{m-1}$.
\end{proof}

\section{Graphs with degeneracy at least 3}\label{d3}
In this section we look at the case when $H$ is a graph with degeneracy at least 3 and prove Theorem \ref{degen3}. We begin with the lower bound. 
\begin{lemma}\label{lemma:linearlowerbound}
Suppose $n\geq m\geq 3$ and $e_1,\cdots,e_k\geq0$ are integers such that $(e_1,\cdots,e_k)$ is $n$-good and
\[\sum_{i=1}^k\binom{e_i}{2}<\frac{n(n-1)(n-2)}{m(m-1)(m-2)}.\]
Then any $(e_1,\cdots,e_k)$-colouring of $K_n$ contains a rainbow copy of $K_m$. 
\end{lemma}
\begin{proof}
Fix an $(e_1,\cdots,e_k)$-colouring of $K_n$. For each $i\in[k]$, let $E_i$ be the set of edges in $K_n$ that has colour $i$. Pick a size $m$ subset $U$ of $[n]$ uniformly at random. For each $i\in[k]$ and any distinct edges $e,e'\in E_i$, the probability that both $e$ and $e'$ are in $U$ is either $\frac{\binom{n-3}{m-3}}{\binom{n}{m}}=\frac{m(m-1)(m-2)}{n(n-1)(n-2)}$ if $e$ and $e'$ share a vertex, or $\frac{\binom{n-4}{m-4}}{\binom{n}{m}}=\frac{m(m-1)(m-2)(m-3)}{n(n-1)(n-2)(n-3)}\leq\frac{m(m-1)(m-2)}{n(n-1)(n-2)}$ if they don't. Using a union bound, we see that the expected number of monochromatic edge pairs in $U$ is at most 
\[\frac{m(m-1)(m-2)}{n(n-1)(n-2)}\sum_{i=1}^k\binom{e_i}{2}<1.\]
Hence, there exists a realisation of $U$ such that the complete graph induced by the $m$ vertices in $U$ is rainbow. 
\end{proof}

\begin{corollary}\label{generallower}
Let $H$ be a graph on $m$ vertices. Then $g(H,k)\geq k/m^3$ for all sufficiently large $k$.
\end{corollary}
\begin{proof}
When $m\le 2$, the theorem is trivial. When $m\ge 3$, let $n=\lfloor k/m^3\rfloor$ and assume $k$ is large enough so that $\binom n2\geq k$. Suppose that $\binom{n}{2}=qk+r$, where $0\le r\le k-1$. Set $e_1=e_2=\cdots=e_{k-r}=q$ and $e_{k-r+1}=\cdots=e_k=q+1$. Then
\begin{align*}
\sum_{i=1}^k\binom{e_i}{2}
&\le \sum_{i=1}^k\frac{1}{2}\left(\frac{\binom{n}{2}}{k}+1\right)\frac{\binom{n}{2}}{k}\\
&\leq\sum_{i=1}^k\frac{\binom{n}{2}^2}{k^2}<\frac{n^4}{4k}\le \frac{n^3}{4m^3}\\
&\le \frac{n(n-1)(n-2)}{m(m-1)(m-2)}.
\end{align*} So by Lemma \ref{lemma:linearlowerbound}, we can find a rainbow copy of $K_m$, and hence a rainbow copy of $H$ in this colouring of $K_n$.
\end{proof}

For the upper bound, we first prove it under the stronger assumption that $H$ has minimum degree at least 3.
\begin{lemma}\label{lemma:mindegree3}
Let $H$ be a graph on $m$ vertices with minimum degree at least 3. Assume $n\geq 2k$ and $(e_1,\cdots,e_k)$ is an $n$-good sequence. Then there exists a rainbow $H$-free $(e_1,\cdots,e_k)$-colouring of $K_n$.  
\end{lemma}
\begin{proof}
We prove this by induction on $k$. Note that every graph $H$ with minimum degree 3 contains at least 6 edges, so any colouring using less than 6 colours will be rainbow $H$-free. Thus the statement is trivially true for all $k\leq5$. Assume now $k\geq 6$.

Without loss of generality, assume $e_1\geq\cdots\geq e_k>0$. Then $e_1\geq\frac1k\binom{n}{2}\geq n-1$ and $e_k\leq\frac1k\binom n2$. Let $t\geq1$ be the smallest integer satisfying $\binom{t}{2}+t(n-t)\geq e_k$. By minimality, $\binom{t-1}{2}+(t-1)(n-t+1)<e_k$, 
and thus 
\[\binom{t}{2}+t(n-t)-e_k<n-t\le e_1.\] This means that we can colour $e_k$ of the $\binom{t}{2}+t(n-t)$ edges adjacent to $v_{n-t+1},\cdots,v_n$ in $K_n$ with colour $k$, and colour the remaining ones with colour 1. 

Also, since $$\binom{\frac nk}{2}+\frac nk\left(n-\frac nk\right)=\frac n{2k}\left(2n-\frac nk-1\right)\geq\frac n{2k}(n-1)=\frac1k\binom{n}2\geq e_k,$$ we have $t\leq\frac nk$ and thus $n-t\geq n(1-\frac1k)\geq 2(k-1)$. Let $e_1'=e_1-\binom t2-t(n-t)+e_k$ and $e_i'=e_i$ for all $2\leq i\leq k-1$. It is easy to verify that $(e_1', e_2', \cdots, e_{k-1}')$ is an $(n-t)$-good sequence. Hence, we can use induction hypothesis to find a rainbow $H$-free $(e_1',\cdots,e_{k-1}')$-colouring of $K_{n-t}$. Combining this with the colouring above for edges adjacent to $v_{n-t+1},\cdots, v_n$, we obtain an $(e_1,\cdots,e_k)$-colouring of $K_n$. 

We claim that this colouring is rainbow $H$-free. Indeed, let $H'$ be any subgraph of $K_n$ isomorphic to $H$. If all vertices of $H'$ are from $v_1,\cdots,v_{n-t}$, then it is not rainbow by the inductive construction. Otherwise, $H'$ contains least one of $v_{n-t+1},\cdots,v_n$. But edges adjacent to each of these vertices can only have colour 1 or $k$, so $H'$ not rainbow as $H'\cong H$ has minimum degree 3. 
\end{proof}

We can now deduce Theorem \ref{degen3}. 
\begin{proof}[Proof of Theorem \ref{degen3}]
Let $H$ be a graph on $m$ vertices with degeneracy at least 3. Then $H$ must contain a subgraph $H'$ with minimum degree at least 3. Lemma \ref{lemma:mindegree3} implies that $g(H',k)\leq 2k$. So by Lemma \ref{subgraph}, $g(H,k)\leq g(H',k)\leq2k$, which gives the upper bound. The lower bound follows from Corollary \ref{generallower}.
\end{proof}

\section{Graphs with degeneracy 2}\label{upper}
In this section, we study the case when $H$ is a graph with degeneracy 2 and prove Theorem \ref{degen2}. We only need to prove the upper bound as the lower bound follows from Corollary \ref{generallower}. Let $H$ be a graph with degeneracy 2. To prove an upper bound of the form $g(H,k)\leq N$, we need to construct a rainbow $H$-free $(e_1,\cdots,e_k)$-colouring of $K_n$ for every $n\geq N$ and every $n$-good sequence $(e_1,\cdots,e_k)$. It turns out that repeated suitable applications of the following colouring steps would work. 

\begin{definition}
Suppose we have a complete graph $K_n$ and a sequence $(e_1,\cdots,e_k)$ satisfying $\sum_{i=1}^ke_i\geq\binom{n}{2}$.
\begin{itemize}
    \item A standard colouring step on $K_n$ involves splitting it into $K_t$ and $K_{n-t}$ for some $1\leq t\leq\floor{\frac n2}$, and colouring all $t(n-t)$ edges between them with some colour $i$ satisfying $e_i\geq t(n-t)$. The size of such a standard colouring step is $t$.
    \item A simple colouring step is a standard colouring step of size 1.
\end{itemize}
\end{definition}

We are interested in colourings of $K_n$ that can be obtained entirely from standard colouring steps. 

\begin{definition}
Given a complete graph $K_n$ and a sequence $(e_1,\cdots,e_k)$ satisfying $\sum_{i=1}^ke_i\geq\binom{n}{2}$, we say $K_n$ has a standard colouring (with respect to $(e_1,\cdots,e_k)$) if it has an edge colouring that can be obtained from a sequence of standard colouring steps. More formally, $K_n$ has a standard colouring if it has an edge colouring that can be obtained using the following stepwise algorithm, where we view each set $S_\ell$ as a multiset. 
\begin{itemize}
    \item Initialise by setting $S_0=\{K_n\}$ and $e_{0,j}=e_j$ for all $j\in[k]$
    \item For $\ell\geq 0$, if $S_\ell=\{K_1,\cdots,K_1\}$, terminate.
    
    Otherwise, pick some $K_m\in S_\ell$ with $m\geq 2$, some $i\in[k]$ and some $1\leq t\leq\floor{\frac m2}$, such that $t(m-t)\leq e_{\ell,i}$.
    \item Set $S_{\ell+1}=(S_\ell\setminus\{K_m\})\cup\{K_t,K_{m-t}\}$, $e_{\ell+1,i}=e_{\ell,i}-t(m-t)$ and $e_{\ell+1,j}=e_{\ell,j}$ for all $j\in[k]\setminus\{i\}$. Note that this correspoinds to a standard colouring step on $K_m$ of size $t$ using colour $i$. 
    \item Repeat the second and third step.
\end{itemize}
Similarly, we say $K_n$ has a simple colouring if it has an edge colouring that can be obtained from a sequence of simple colouring steps. 
\end{definition}

We first prove a lemma that explains why we are interested in standard colouring steps and standard colourings. 
\begin{lemma}\label{standardhfree}
Every standard colouring of $K_n$ is rainbow cycle-free, and therefore rainbow $H$-free for all graphs $H$ with degeneracy 2.
\end{lemma}
\begin{proof}
For every $3\leq k\leq n$, let $C$ be an arbitrary cycle of length $k$ in $K_n$. Let $\ell\geq1$ be the smallest index such that vertices in $C$ do not all belong to the same complete graph in $S_\ell$. By minimality, they all belong to the same complete graph in $S_{\ell-1}$, say $K_m$, that was split into two complete graphs $K_t$ and $K_{m-t}$ by a standard colouring step in the $\ell$-th iteration of the algorithm. From definition, all edges between $K_t$ and $K_{m-t}$ have the same colour. Since at least two edges on $C$ go between $K_t$ and $K_{m-t}$, $C$ is not rainbow. 
\end{proof}

Assume now that the algorithm starts with a complete graph $K_n$ and an $n$-good sequence $(e_1,\cdots,e_k)$. Then throughout the algorithm, there is a relation between the sizes of the complete graphs in $S_{\ell}$ and the $e_{\ell,j}$'s.
\begin{lemma}\label{relation}
For every $\ell\geq 0$, if $S_\ell=\{K_{m_1},\cdots,K_{m_r}\}$, then 
\[\sum_{j=1}^r\binom{m_j}{2}=\sum_{j=1}^ke_{\ell,j}.\]
\end{lemma}
\begin{proof}
We use induction on $\ell$. The case $\ell=0$ follows from the assumption that $(e_1,\cdots,e_k)=(e_{0,1},\cdots,e_{0,k})$ is $n$-good. Suppose this is true for some $\ell\geq0$, and without loss of generality assume that the $(\ell+1)$-th standard colouring step is splitting $K_{m_1}$ into $K_t$ and $K_{m_1-t}$ and colouring the $t(m_1-t)$ edges with colour $1$. Then we have 
\begin{align*}
\binom{t}{2}+\binom{m_1-t}{2}+\sum_{j=2}^r\binom{m_j}{2}&=\binom{t}{2}+\binom{m_1-t}{2}-\binom{m_1}{2}+\sum_{j=1}^ke_{\ell,j}\\
&=-t(m_1-t)+e_{\ell,1}+\sum_{j=2}^ke_{\ell,j}\\
&=\sum_{j=1}^ke_{\ell+1,j},
\end{align*}
as required.
\end{proof}

This relation can also be proved by noting that edges that remained to be coloured after $\ell$ iterations are exactly those within the complete graphs in $S_\ell$.

Using this, we can prove that a simple colouring step can always be performed during the algorithm under a mild size condition. 
\begin{lemma}\label{remove1}
For every $\ell\geq 0$, if $S_\ell=\{K_{m_1},\cdots,K_{m_r}\}$ and $m_1\geq 2k$, then a simple colouring step can be performed on $K_{m_1}$.
\end{lemma}
\begin{proof}
By Lemma \ref{relation}, we have $$\sum_{j=1}^ke_{\ell,j}=\sum_{j=1}^r\binom{m_j}{2}\geq\binom{m_1}{2}.$$ So there exists $i\in[k]$ such that $e_{\ell,i}\geq\frac1{k}\binom{m_1}{2}\geq m_1-1$. Hence, we can perform a simple colouring step on $K_{m_1}$ using colour $i$.
\end{proof}

To obtain result analogous to Lemma \ref{remove1} below size $2k$, we need the following important concept of cushion.
\begin{definition}
For $\ell\geq 0$, if $S_\ell=\{K_{m_1},\cdots,K_{m_r}\}$, then we say the cushion we have for $K_{m_1}$ is $\sum_{j=1}^ke_{\ell,j}-\binom{m_1}{2}=\sum_{j=2}^r\binom{m_j}{2}$, and we say that this cushion is provided by $K_{m_2},\cdots,K_{m_r}$.
\end{definition}

\begin{lemma}\label{2kspare}
For all $m\geq 2$, suppose $K_m\in S_{\ell}$ and the cushion we have for $K_m$ is at least $\min\{\frac12(k^2-k),km\}$. Then $m-1$ consecutive simple colouring steps can be performed on $K_m$ to produce a simple colouring.
\end{lemma}
\begin{proof}
It suffices to show that for each $0\leq t\leq m-2$, it is possible to perform a simple colouring step on $K_{m-t}$. Indeed, performing simple colouring steps on $K_{m},\cdots,K_{m-t+1}$ colours $\sum_{j=1}^{t}(m-j)=\frac12t(2m-t-1)$ edges in total. 

If the cushion we have is at least $\frac12(k^2-k)$, then 
\begin{align*}
\sum_{j=1}^ke_{\ell+t,j}&=\sum_{j=1}^ke_{\ell,j}-\frac12t(2m-t-1)\\
&\geq\binom{m}{2}+\frac12(k^2-k)-\frac12t(2m-t-1)\\
&=\frac12(m^2-m+k^2-k-2mt+t^2+t)\\
&=\frac12(t-m+k)(t-m+k+1)+k(m-t-1)\\
&\geq k(m-t-1).
\end{align*}

If instead the cushion we have is at least $km$, then 
\begin{align*}
\sum_{j=1}^ke_{\ell+t,j}&=\sum_{j=1}^ke_{\ell,j}-\frac12t(2m-t-1)\\
&\geq\binom{m}{2}+km-\frac12t(2m-t-1)\\
&=\binom{m-t}{2}+km\\
&\geq km\geq k(m-t-1).
\end{align*}
Hence, in both cases there exists a colour $i$ with $e_{\ell+t,i}\geq m-t-1$. So we can perform a simple colouring step on $K_{m-t}$ using colour $i$, as required.
\end{proof}
We remark that a weaker version of this lemma in the case when $m=2k$ was proved by Feffer, Fu and Yan in \cite{FFY} using a much more complicated method.

We prove one final lemma before proving Theorem \ref{degen2}. This lemma allow us to perform several consecutive standard colouring steps of the same size, which will be useful for creating cushion. 

\begin{lemma}\label{cr}
Let $n,m,t,k\geq 1$ and $e_1,\cdots,e_k\geq 0$ be integers satisfying $n>tm$ and $\sum_{i=1}^ke_i>tn(m+k)$. Then we can perform $m$ consecutive standard colouring steps of size $t$ on $K_n$.
\end{lemma}
\begin{proof}
For each $j\in[m]$, the $j$-th size $t$ standard colouring step splits $K_{n-(j-1)t}$ into $K_t$ and $K_{n-jt}$, and colour the $t(n-jt)$ edges between them using one of the $k$ colours. Since $t(n-jt)\leq tn$, colour $i$ can be used to perform a size $t$ standard colouring step at least $\floor{\frac{e_i}{tn}}$ times for each $i\in[k]$. Since $$\sum_{i=1}^k\floor{\frac{e_i}{tn}}\geq\sum_{i=1}^k\frac{e_i}{tn}-k>m+k-k=m,$$ it is possible to perform $m$ such steps. 
\end{proof}

We now sketch the upper bound proof in Theorem \ref{degen2}, which proceeds in three stages:

\begin{itemize}
\item Stage 1: We perform $k$ consecutive standard colouring steps of size $r=\Theta\left(\frac{k^{0.5}}{(\log k)^{0.5}}\right)$ using Lemma \ref{cr}. Let $\mathcal{R}$ be the set of $k$ complete graphs $K_r$ split off in these steps. These complete graphs in $\mathcal{R}$ will provide a small but still significant amount of cushion. 

\item Stage 2: We perform several standard colouring steps of larger size, so that the set $\mathcal{C}$ of complete graphs split off in this stage provide a big cushion of at least $\frac12k^2$.

\item Stage 3: By Lemma \ref{remove1}, we can reduce the size of the largest remaining complete graph down to $K_{2k}$. Using the cushion provided by $\mathcal{C}$ and Lemma \ref{2kspare}, we can find a simple colouring for this $K_{2k}$. Then, we use Lemma \ref{2kspare} again and the cushion provided by $\mathcal{R}$ to find simple colourings for each complete graph in $\mathcal{C}$. Finally, we show that the cushion provided by $\mathcal{R}$ is enough for us to find simple colourings for the complete graphs in $\mathcal{R}$ itself as well.


\end{itemize}

The idea of this proof is already present implicitly in the upper bound proof in \cite{FFY} but with $\mathcal{C}$ and $\mathcal{R}$ being the same set, which turns out to be less than optimal.

Finally, we can prove Theorem \ref{degen2}.

\begin{proof}[Proof of Theorem \ref{degen2}]

The lower bound follows directly from Corollary \ref{generallower}, so we just need to prove the upper bound. The proof proceeds in 3 stages. Assume $k$ is sufficiently large and set $$\beta=5000000,\text{ }n=\floor{\frac{\beta k^{1.5}}{(\log{k})^{0.5}}}.$$
Let $(e_1,\cdots,e_k)$ be any $n$-good sequence, and write $e=\binom{n}{2}$. Initialise the colouring algorithm by setting $e_{0,j}=e_j$ for all $j\in[k]$ and $S_0=\{K_n\}$. 

\vspace{1em}

\noindent\textbf{Stage 1: The Small Cushion Steps}

\vspace{1em}

Set $r=\floor{\frac{\beta k^{0.5}}{30(\log k)^{0.5}}}$. Since $$\sum_{j=1}^ke_{0,j}=e=\binom{n}{2}>\frac{\beta^2k^3}{3\log k}\geq10krn>rn(k+k),$$ we may perform $k$ consecutive standard colouring steps of size $r$ on $K_n$ using Lemma \ref{cr}. Let $\mathcal{R}$ be this collection of $k$ complete graphs of size $r$. Note that we have coloured at most $krn\leq\frac{\beta^2k^3}{30\log k}<0.1e$ edges. Therefore, at the end of this stage, we have $\sum_{j=1}^ke_{k,j}\geq0.9e$.

\vspace{1em}

\noindent\textbf{Stage 2: The Big Cushion Steps}

\vspace{1em}

Let $J_1=\left\{j\in[k]\mid e_{k,j}\geq\frac{\beta^2k^{2.25}}{(\log k)^{1.25}}\right\}$. Since \[|J_1|\frac{\beta^2k^{2.25}}{(\log k)^{1.25}}\leq\sum_{j\in J_1}e_{k,j}\leq e,\] it follows that $|J_1|\leq\frac12k^{0.75}(\log k)^{0.25}$. We split into two cases.

\vspace{1em}

\textbf{Case 1.} $\sum_{j\in J_1}e_{k,j}\geq0.1e$.

\vspace{1em}

Set $c=\floor{\frac{k^{0.75}}{(\log k)^{0.75}}}$. Since \[\sum_{j\in J_1}e_{k,j}\geq0.1e\geq\frac{\beta^2k^3}{30\log k}\geq cn\cdot\frac{\beta k^{0.75}(\log k)^{0.25}}{30}>cn(\ceil{k^{0.75}(\log k)^{0.25}}+|J_1|),\]
we can perform $\ceil{k^{0.75}(\log k)^{0.25}}$ consecutive simple colouring steps of size $c$ using only the colours in $J_1$ by Lemma \ref{cr}. Let $\mathcal{C}$ be this collection of $\ceil{k^{0.75}(\log k)^{0.25}}$ complete graphs of size $c$. Observe that complete graphs in $\mathcal{C}$ create a cushion of $$k^{0.75}(\log k)^{0.25}\binom{c}{2}\geq\frac{k^{2.25}}{3(\log k)^{1.25}}\gg k^2,$$
which is enough for us to apply Lemma \ref{2kspare} later on. At this point, the set of complete graphs yet to be coloured consists of the complete graphs in $\mathcal{R}$ coming from Stage 1, the complete graphs in $\mathcal{C}$ coming from Stage 2 and one other large complete graph of size $n-rk-c\ceil{k^{0.75}(\log k)^{0.25}}\gg2k$.

\vspace{1em}

\textbf{Case 2. }$\sum_{j\in J_1}e_{k,j}<0.1e$.

\vspace{1em}

Let $J_2=\left\{j\in[k]\mid e_{k,j}\leq\frac{\beta^2k^2}{30\log k}\right\}$, then $\sum_{j\in J_2}e_{k,j}\leq\frac{\beta^2k^3}{30\log k}\leq0.1e$. Let $J_3=[k]\setminus(J_1\cup J_2)$. Without loss of generality, assume $J_3=[\ell]$. Then $\frac{\beta^2k^2}{30\log k}<e_{k,j}<\frac{\beta^2k^{2.25}}{(\log k)^{1.25}}$ for all $j\in J_3=[\ell]$ and \[\sum_{j=1}^\ell e_{k,j}=\sum_{j=1}^ke_{k,j}-\sum_{j\in J_1}e_{k,j}-\sum_{j\in J_2}e_{k,j}\geq0.9e-0.1e-0.1e=0.7e.\]

For each $j\in J_1\cup J_2$, use colour $j$ to perform as many simple colouring steps as possible to the largest remaining complete graph. Suppose a total of $N-k$ simple colouring steps are performed. Then for each $j\in J_1\cup J_2$, we have $e_{N,j}<n$, as otherwise we could perform another simple colouring step using colour $j$. Let $K_{x_0}$ be the largest complete graph in $S_{N}$. Then $S_{N}$ consists of $K_{x_0}$, $k$ complete graphs $K_r$ coming from Stage 1, and many $K_1$ from the simple colouring steps above. Therefore, by Lemma \ref{relation}, we have $$\binom{x_0}{2}+k\binom{r}{2}=\sum_{j\in J_1\cup J_2}e_{N,j}+\sum_{j=1}^\ell e_{N,j}\geq\sum_{j=1}^\ell e_{N,j}=\sum_{j=1}^\ell e_{k,j}\geq0.7e.$$
Since $k\binom{r}{2}\ll e$, it follows that $0.5x_0^2\geq\binom{x_0}{2}\geq0.66e\geq0.32\frac{\beta^2k^3}{\log k}$, and thus $x_0\geq\frac{4\beta k^{1.5}}{5(\log k)^{0.5}}$.

Now consider the following process. 

\begin{itemize}
\item Start with the complete graph $K_{x_0}$. 

\item In the $j$-th step, if $x_{j-1}<\frac{2\sqrt{\beta}k^{1.25}}{(\log k)^{0.25}}$, we stop. Otherwise,  perform a standard colouring step on $K_{x_{j-1}}$ of maximum possible size $c_j<\frac12x_{j-1}$ to split $K_{x_{j-1}}$ into $K_{x_j}$ and $K_{c_j}$, where $x_j=x_{j-1}-c_j$.
\end{itemize}
The maximality of $c_j$ implies that 
\[c_jx_j\leq e_{N+j-1,j} \text{\quad and \quad} (c_j+1)(x_j-1)>e_{N+j-1,j}.\]
From this, it follows that $$e_{N+j,j}=e_{N+j-1,j}-c_jx_j<x_j-c_j-1\leq x_j-1< n,$$ 
and \[\frac{e_{N+j-1,j}}{2x_j}\leq\frac{e_{N+j-1,j}}{x_j-1}-1\leq c_j\leq\frac{e_{N+j-1,j}}{x_j}, \tag{$\spadesuit$}\]
where the first inequality holds because
\[e_{N+j-1,j}=e_{k,j}\ge \frac{\beta^2k^2}{30\log k}>2n\ge 2x_j. \]

Assume the process stop in step $\ell'$. We claim that $x_{\ell'}\le\frac{2\sqrt{\beta}k^{1.25}}{(\log k)^{0.25}}$. Indeed, this follows trivially if the process stops because $x_{j-1}<\frac{2\sqrt{\beta}k^{1.25}}{(\log k)^{0.25}}$ for some $j\in[\ell]$. If instead we performed all $\ell$ standard colouring steps described above, then at the end we have $$\sum_{j=1}^ke_{N+\ell,j}=\sum_{j\in J_1\cup J_2}e_{N,j}+\sum_{j=1}^\ell e_{N+j,j}< kn.$$
This implies that $x_{\ell}<\frac{2\sqrt{\beta}k^{1.25}}{(\log k)^{0.25}}$, as required, because otherwise \[\binom{x_{\ell}}{2}>kn\geq\sum_{j=1}^ke_{N+\ell,j},\]
contradicting Lemma \ref{relation}.

Since  $x_0\geq\frac{4\beta k^{1.5}}{5(\log k)^{0.5}}$, we have 
$\frac{x_0}{x_{\ell'}}\geq \frac{2\sqrt{\beta}k^{0.25}}{5(\log k)^{0.25}}$. Moreover, we have 
\[2k\ll\frac{\sqrt{\beta}k^{1.25}}{(\log k)^{0.25}}\leq \frac12x_{\ell'-1}\leq x_{\ell'}<\frac{2\sqrt{\beta}k^{1.25}}{(\log k)^{0.25}},\] 
It follows that for $j\in[\ell']$, we have 
\[c_j\leq\frac{e_{N+j-1,j}}{x_j}\leq\frac{\beta^2k^{2.25}}{(\log k)^{1.25}}\frac{(\log k)^{0.25}}{\sqrt{\beta}k^{1.25}}\leq\frac{\beta^{1.5}k}{\log k}. \tag{$\clubsuit$}\]

Let $\mathcal{C}$ be the collection $K_{c_1},\cdots,K_{c_{\ell'}}$ of small complete graphs split off in this process. The amount of cushions provided by $\mathcal{C}$ is at least

\begin{align*}
    \sum_{j=1}^{\ell'}\binom{c_j}{2}\geq\frac13\sum_{j=1}^{\ell'}c_j^2
    &\geq\frac13\sum_{j=1}^{\ell'}c_j\frac{e_{N+j-1,j}}{2x_j} && \text{by $(\spadesuit)$}\\
    &=\frac13\sum_{j=1}^{\ell'}c_j\frac{e_{k,j}}{2x_j}\\
    &\geq\frac{\beta^2k^2}{180\log k}\sum_{j=1}^{\ell'}\frac{c_j}{x_j} && \text{by }e_{k,j}\ge \frac{\beta^2k^2}{30\log k}\\
    &=\frac{\beta^2k^2}{180\log k}\sum_{j=1}^{\ell'}\int_{x=x_j}^{x_{j-1}}\frac{1}{x_j}\text{d}x &&\text{by }c_j=x_{j-1}-x_j\\
    &\geq\frac{\beta^2k^2}{180\log k}\sum_{j=1}^{\ell'}\int_{x=x_j}^{x_{j-1}}\frac{1}{x}\text{d}x\\
    &=\frac{\beta^2k^2}{180\log k}\int_{x=x_{\ell'}}^{x_0}\frac 1x\text{d}x\\
    &=\frac{\beta^2k^2}{180\log k}\log{\frac{x_0}{x_{\ell'}}}\\
    &\geq\frac{\beta^2k^2}{180\log k}\log\frac{2\sqrt{\beta}k^{0.25}}{5(\log k)^{0.25}}\\
    &\geq\frac{\beta^2k^2}{720\log k}(\log k-\log\log k)\gg k^2,
\end{align*}
which is enough for us to apply Lemma \ref{2kspare} later on. At this point, the set of complete graphs left to be coloured consists of those in $\mathcal{R}$, those in $\mathcal{C}$ and the complete graph $K_{x_{\ell'}}$ with $x_{\ell'}\gg2k$. 

\vspace{1em}

\noindent\textbf{Stage 3: Colour the Remaining Complete Graphs}

\vspace{1em}

In both cases, there is only one complete graph left with size at least $2k$. We can first use Lemma \ref{remove1} to perform as many simple colouring steps on it as needed to reduce its size down to $2k$. Then we can use Lemma \ref{2kspare} to find a simple colouring for it, because the complete graphs in $\mathcal{C}$ provide a cushion of at least $\frac12k^2$.

It remains to colour the complete graphs in $\mathcal{C}$ and $\mathcal{R}$. First we find a simple colouring for each of those in $\mathcal{C}$. Recall that, for every $K_s\in \mathcal{C}$, 
\begin{itemize}
\item If we are in Case 1, then $s=c=\floor{\frac{k^{0.75}}{(\log k)^{0.75}}}$,
\item If we are in Case 2, then $s=c_j\leq\frac{\beta^{1.5}k}{\log k}$ for some $j\in[\ell']$ by $(\clubsuit)$.
\end{itemize}
Note that the $k$ complete graphs of size $r$ in $\mathcal{R}$ together provide a cushion of at least \[k\binom{r}{2}\geq\frac{\beta^2k^2}{2000\log k}\geq ks\] for $K_s$. Hence, by Lemma \ref{2kspare} we can find a simple colouring for $K_s$.

Finally, we prove that we can colour all $k$ complete graphs of size $r$ in $\mathcal{R}$. We do this by showing that it is always possible to perform a simple colouring step on the largest remaining complete graphs in $\mathcal{R}$. Indeed, suppose the largest remaining complete graph has size $2\leq t\leq r$, then the other $k-1$ complete graphs all have size at least $t-1$. Hence by Lemma \ref{relation}, we have \[\sum_{j=1}^ke_{L,j}\geq(k-1)\binom{t-1}{2}+\binom{t}{2}\geq(k-1)(t-2)+(t-1)=k(t-2)+1.\]
So there exists some $i\in[k]$ with $e_{L,i}\geq t-1$, which we can use to perform a simple colouring step on $K_t$. As this is true for every $t$, we can repeat and colour all complete graphs in $\mathcal{R}$. This completes a standard colouring of $K_n$ with colour distribution sequence $(e_1,\cdots,e_k)$. By Lemma \ref{standardhfree}, this colouring is rainbow $H$-free, which completes the proof.
\end{proof}

\section{Triangle}\label{t}
In this section, we resolve the question posed by Gy\'arf\'as et al. in \cite{GPPW} by proving Theorem \ref{triangle}. Recall that rainbow $K_3$-free colourings are also known as Gallai colourings, and $g(K_3,k)=g(k)$. As we have proved the upper bound in Section \ref{upper}, it suffices to prove the lower bound here. Note that Theorem \ref{triangle} also shows that the upper bound in Theorem \ref{degen2} is tight up to the choice of constant in the case when $H=K_3$. 

We will use the following fundamental decomposition theorem of Gallai colourings.
\begin{theorem}[\cite{FFY}]\label{strongdecomp}
Suppose we have a Gallai $k$-colouring of $K_n$. Then there exist at most two colours, which we call base colours, and a decomposition of $K_n$ into $m\ge 2$ vertex disjoint complete graphs $K_{n_1},\cdots,K_{n_m}$, such that for each $1\leq i<j\leq m$, there exists a base colour with all edges between $V(K_{n_i})$ and $V(K_{n_j})$ having this colour. Moreover, each base colour is used to colour at least $n-1$ edges between these smaller complete graphs. 
\end{theorem}

Now we can prove Theorem \ref{triangle}.

\begin{proof}[Proof of Theorem \ref{triangle}] By Theorem \ref{degen2}, we only need to show  $g(H,k)=\Omega(k^{1.5}/(\log k)^{0.5}).$ The proof consists of 3 steps.

\vspace{1em}

\noindent{\bf Step 1: Construct the sequence and describe a splitting process}

\vspace{1em}

Consider the colour distribution sequence $(e_1,\cdots,e_k)$ on $n$ vertices given by $e_1=\cdots=e_{c}=a+1$, $e_{c+1}=\cdots=e_{\ceil{\frac{k}{2}}}=a$ and $e_{\ceil{\frac{k}{2}}+1}=\cdots=e_{k}=b$, where
$$\alpha=\frac1{10},\text{ } n=\floor{\frac{\alpha k^{1.5}}{(\log k)^{0.5}}},\text{ } b=\floor{\frac{k}{2}},\text{ } a=\floor{\frac{\binom{n}{2}-b\floor{\frac{k}{2}}}{\ceil{\frac{k}{2}}}},\text{ } c=\binom{n}{2}-b\floor{\frac{k}{2}}-a\ceil{\frac{k}{2}}.$$ 
It is easy to verify that $(e_1,e_2,\cdots,e_k)$ is an $n$-good sequence. We will show that for sufficiently large $k$, there is no Gallai colouring of $K_n$ with colour distribution sequence $(e_1,\cdots,e_k)$. 

Suppose for a contradiction that there is such a Gallai colouring, then we can perform the following stepwise splitting process by repeatedly applying Theorem \ref{strongdecomp}:
\begin{itemize}
\item Initialise with $K_{x_0}=K_n$.

\item For each $i\geq0$, if $x_i\le b+1$, we stop. Otherwise, by Theorem \ref{strongdecomp}, we can find a decomposition of $K_{x_i}$ using at most two base colours into $m\ge 2$ vertex disjoint complete graphs $K_{y_1},\cdots,K_{y_m}$, where $y_1\geq\cdots\geq y_m$.

\item Let $t_{i+1}=y_m$ and $x_{i+1}=x_i-t_{i+1}$. Split $K_{x_i}$ into $K_{y_m}$ and $K_{x_{i+1}}$, with the latter being the union of $K_{y_1},\cdots,K_{y_{m-1}}$. Note that by Theorem \ref{strongdecomp}, every edge between $V(K_{x_{i+1}})$ and $V(K_{y_m})$ is coloured with a base colour.

\item Repeat the second and third step.

\end{itemize}

Hence, we can find a nested sequence of complete graphs $K_{x_0}\supset K_{x_1}\supset\cdots\supset K_{x_\ell}$, where $n=x_0>x_1>\cdots>x_{\ell-1}\ge b+2>x_{\ell}$. Moreover, by the last part of Theorem \ref{strongdecomp}, in each iteration of the process above, the base colour chosen must be used to colour at least $b+1$ edges. Hence colours $\ceil{\frac k2}+1,\cdots,k$ will never be picked as base colours in this process.

\vspace{1em}

\noindent{\bf Step 2: Lower bound on the number of edges coloured with base colours}

\vspace{1em}

We claim that $t_{i+1}\leq\frac{4(a+1)}{x_{i}}$ for all $0\leq i\leq\ell-1$. Indeed, every edge between the complete graphs $V(K_{x_{i+1}})$ and $V(K_{y_m})$ must be coloured with a base colour. However, the at most 2 base colours together can be used to colour at most $2(a+1)$ edges. This implies that
\[\frac12x_it_{i+1}\leq(x_i-t_{i+1})t_{i+1}=x_{i+1}y_m\le 2(a+1),\]
where we used $t_{i+1}=y_m\leq\frac12x_i$. Thus, $t_{i+1}\leq\frac{4(a+1)}{x_i}$, as claimed.

We can now give a lower bound on the number of edges coloured with base colours in this splitting process. This number is
\begin{align*}
\sum_{i=1}^{\ell}t_i(x_{i-1}-t_i)&=\sum_{i=1}^{\ell}\int_{x=x_i}^{x_{i-1}}(x_{i-1}-t_i)\text{d}x\\
&\geq\sum_{i=1}^{\ell}\int_{x=x_i}^{x_{i-1}}\left(x_{i-1}-\frac{4(a+1)}{x_{i-1}}\right)\text{d}x\\
&\geq\sum_{i=1}^{\ell}\int_{x=x_i}^{x_{i-1}}\left(x-\frac{4(a+1)}{x}\right)\text{d}x\\
&=\int_{x=x_{\ell}}^{x_0}\left(x-\frac{4(a+1)}{x}\right)\text{d}x\\
&\geq\int_{x=b+1}^{n}\left(x-\frac{4(a+1)}{x}\right)\text{d}x\\
&=\frac{n^2-(b+1)^2}{2}-4(a+1)\log\frac{n}{b+1}\\
&\geq\frac{n^2}{2}-\frac{2b^2}{3}-4(a+1)\log\frac n{b}.
\end{align*}

Recall that only the first $\ceil{\frac{k}{2}}$ colours can be used as base colours in this process.

\vspace{1em}

\noindent{\bf Step 3: Not enough edges can be coloured with the first $\ceil{\frac k2}$ colours}

\vspace{1em}

The total number of edges that can be coloured with the first $\ceil{\frac{k}{2}}$ colours is
$$\sum_{i=1}^{\ceil{\frac k2}}e_i=\binom{n}{2}-b\floor{\frac k2}\leq\frac{n^2}{2}-b^2.$$
So we will have a contradiction if $$\frac{n^2}{2}-\frac{2b^2}{3}-4(a+1)\log\frac n{b}>\frac{n^2}{2}-b^2.$$

Note that for sufficiently large $k$ we have 
$$b^2=\floor{\frac{k}{2}}^2\geq\left(\frac{k-1}2\right)^2\geq\frac{k^2}{5},$$ 
$$4(a+1)\leq 5a\leq 5\frac{\binom{n}{2}}{\ceil{{\frac k2}}}\leq\frac{5n^2}{k}\leq\frac{5\alpha^2k^2}{\log k},$$ 
$$\log\frac nb\leq\log\left(\frac{\alpha k^{1.5}}{(\log k)^{0.5}}\frac{3}{k}\right)=\log\left(\frac{3\alpha k^{0.5}}{(\log k)^{0.5}}\right)=\frac12(\log k-\log\log k)+\log(3\alpha)\leq\log k.$$
Putting these together and using $\alpha=\frac1{10}$, we get $$\left(\frac{n^2}{2}-\frac{2b^2}{3}-4(a+1)\log\frac n{b}\right)-\left(\frac{n^2}{2}-b^2\right)=\frac{b^2}3-4(a+1)\log\frac nb\geq\frac{k^2}{15}-\frac{k^2}{20}>0,$$
as required. This proves that if $k$ is sufficiently large, then there is no Gallai colouring with colour distribution sequence $(e_1,\cdots,e_k)$. Therefore, $g(k)\geq\frac{\alpha k^{1.5}}{(\log k)^{0.5}}$. 
\end{proof}

\end{document}